\documentclass[11pt]{amsart}

\usepackage{hyperref}
\usepackage[utf8]{inputenc}
\usepackage[english]{babel}
\usepackage[T1]{fontenc}
\usepackage{amsthm}
\usepackage{amssymb}
\usepackage{amsmath}
\usepackage{mathtools}
\usepackage{pdfpages}
\usepackage{changes}
\usepackage{tikz-cd}
\usepackage{soul}

\newcommand{\Z}{\mathbb{Z}}

\newcommand{\GL}{\mathrm{GL}}
\newcommand{\id}{\mathrm{id}}

\makeatletter
\numberwithin{equation}{section}
\numberwithin{figure}{section}
\numberwithin{table}{section}
\newtheorem{theorem}{Theorem}[section]

\newtheorem{corollary}[theorem]{Corollary}
\newtheorem{proposition}[theorem]{Proposition}
\newtheorem{definition}[theorem]{Definition}

\newtheorem{conjecture}[theorem]{Conjecture}

\newtheorem{remark}[theorem]{Remark}

\makeatother

\def\Z{{\mathbb{Z}}}

\begin{document}


\title[Indecomposable solutions of the Yang-Baxter equation]{Indecomposable solutions of the Yang-Baxter equation with permutation
      group of sizes $pq$ and $p^2q$}
\author{S. Ram\'irez}

\address{IMAS-CONICET and Depto. de Matem\'atica, FCEN, Universidad de Buenos Aires, Pabell\'on~1,
Ciudad Universitaria, C1428EGA, Buenos Aires, Argentina}
\email{sramirez@dm.uba.ar}

\begin{abstract}
In this paper we study the problem of classification of indecomposable solutions
of the Yang-Baxter equation. Using a scheme proposed by Bachiller, Cedó, and
Jespers, and recent advances in the classification of braces we classify all
indecomposable solutions with some particular permutation groups. We do this
for all groups of size $pq$, all abelian groups of size $p^2q$ and all dihedral
groups of size $p^2q$.
\end{abstract}

\maketitle

\section{Introduction}

The study of solutions of the Yang-Baxter equation originates from the works in
physics of Yang
\cite{yang1967many-body} and Baxter \cite{baxter1972partition}. In 1990 Drinfeld
\cite{drinfeld1990original} proposed the study of the class of set-theoretical
solutions of Yang-Baxter equation (see Definition~\ref{def:solution}).

The foundational works on the algebraic study of these solutions came later in
the works of Etingof, Schedler and Soloviev \cite{ESS1999set-theoretical} and
Gateva-Ivanova and Van den Bergh \cite{gateva-ivanova1998I-type}. In the years
following a great number of connections to other areas of algebra have been found,
e.g. Hopf-Galois structures \cite{vendramin2018skew-braces}, radical rings
\cite{rump2007braces} and Garside groups
\cite{chouraqui2010garside,dehornoy2015yang-baxter}.

One of main problems is that of classification and construction of solutions.
Although there have been recent improvements in the explicit construction of all
solutions of small sizes \cite{vendramin2022enumeration}, this approach does not
look feasible in general. There are for example
almost five million involutive (see Definition~\ref{def:involutive}) solutions of
size 10, the largest size computed. There seems to be a computational limit
for direct calculations.

One might try considering a class of simpler solutions from which one might construct
all the rest, and try to classify these simpler ones. One such approach is to
consider solutions that cannot be written as a disjoint union of two other
solutions. These \emph{indecomposable} solutions were defined originally in
\cite{ESS1999set-theoretical}. Classifying indecomposable solutions seems to be
a more approachable objective, there are, for example, only 36 indecomposable
solutions of size 10. This class of solutions has been intensively studied
with many recent results
\cite{MR4388351,MR4330443,MR4308636,MR4163866,MR4116644,MR4085769,MR3958100,MR1848966,
vendramin2021decomposition,camp2021criterion}
In this paper we give a classification of all indecomposable solutions whose permutation
group (see Definition~\ref{def:involutive}) has size $pq$, or
is abelian or dihedral and of size $p^2q$, where $p$ and $q$ are distinct primes.

\section{Preliminaries}

In this section we collect the basic definitions and results we will need about
solutions and braces, and go over the scheme proposed by Bachiller, Cedó and Jespers
\cite{BCJ2016solutions} that we will use for the classification.

\subsection{Solutions and braces}
We will restrict ourselves to non-degenerate solutions, defined in the following
way:
\begin{definition}\label{def:solution}
A \emph{non-degenerate solution to the Yang-Baxter equation} is a pair $(X,r)$,
with $X$ a non-empty set and $r:X\times X\to X\times X$ a bijection that satisfies
\[
(r\times 1)(1\times r)(r\times 1) = (1\times r)(r\times 1)(1\times r),
\]
and such that all the maps $\sigma_x,\tau_y$ defined by $r(x,y)=(\sigma_x(y),\tau_y(x))$
are bijections.
\end{definition}
Besides this we will also restrict ourselves to involutive solutions.
\begin{definition}\label{def:involutive}
A solution of the set-theoretical Yang-Baxter $(X,r)$ is called \emph{involutive}
if $r^2=\id$.

In this case the group generated by the maps $\sigma_x$ is called the
\emph{permutation group} of the solution.
\end{definition}
In what follows \emph{solution} will always mean non-degenerate involutive solution
of the Yang-Baxter equation. These restrictions guarantee that the permutation
group will be well-behaved and nicely reflect properties of the solution. With
these restrictions the indecomposables solutions described in the Introduction
can be alternatively defined in the following way:
\begin{definition}\label{def:indecomposable}
A solution $(X,r)$ is called indecomposable if its permutation group acts transitevely
on $X$.
\end{definition}

The permutation group of a (involutive) solution has the additional structure of
a brace.
\begin{definition}\label{def:brace}
A \emph{brace} is a triple $(B,+,\circ)$, with $(B,+)$ an abelian group,
the additive structure of the brace, and $(B,\circ)$ a group, the multiplicative
structure, satisfying
\[
a\circ(b+c) = a\circ b - a + a\circ c.
\]

A brace is called \emph{trivial} if the additive and multiplicative structures
are the same.
\end{definition}
This structure was originally described in \cite{rump2007braces}, with this
definition first appearing \cite{cedo2014braces}, and is a
fundamental tool in the classification process. The multiplicative structure of
the brace associated to a solution is the usual composition of permutations. We
will not find it necessary to know how to give the additive structure.

In a brace, $(B,+,\circ)$, the multplicative group has a natural action (by group
automorphisms) on the additive group. This action is usually noted by
$\lambda_{-}:(B,\circ)\to \operatorname{Aut}(B,+)$, and is given by
\[
\lambda_{a}(b) = -a+ a\circ b.
\]

\subsection{Classification scheme}
We now go over the scheme developed by Bachiller, Cedó and Jespers
in \cite{BCJ2016solutions} for classifying all solutions that have a given brace
as permutation group. A related scheme in terms of coverings of solutions was
developed by Rump in \cite{MR4116644}.

To give a solution
with brace $B$ one must first choose a set of elements of $B$ such that their
orbits under the $\lambda$ action additively generate $B$. Then for each of
these elements one must choose a subgroup of its stabilizer, in such a way that
the intersection of all of their normal cores is trivial. The solution set is
then the disjoint union of the quotient by these subgroups.
By \cite[Theorem 3.1]{BCJ2016solutions} all solutions can be constructed this way.

On the other hand in \cite[Theorem 4.1]{BCJ2016solutions} they characterize when
two of these solutions are isomorphic. For this to happen there must be a brace
automorphism that maps the orbits of the elements chosen for the first to the orbits
of the elements chosen for the second one. Moreover, the corresponding chosen
subgroups must be conjugate.

Since we are going to be focused on indecomposable solutions this scheme simplifies
considerably. Since indecomposable solutions are precisely those in which the
permutation group acts transitively, in the construction we must choose a
single element. So rather than consider all possible ways in which a union of
orbits may generate the group,
we need only understand which orbits generate the group. The following observation
will be particularly useful for this:
\begin{remark}\label{rmk:orbit-order}
Since all the elements of an orbit have the same additive order if that orbit
generates, the order of these elements must be divisible by all the primes dividing
the size of the group.
\end{remark}

A second important consequence of choosing a single element is that the intersection
we need to consider contains a single subgroup. We are then restricted to choosing
core-free subgroups of the stabilizer of an element. This further simplifies when
considering abelian groups, as in this case the only core-free subgroup is the
trivial one. Moreover, it turns out that even in the case of the dihedral groups
this is the only subgroup we need to consider.

We can then restate the results of \cite{BCJ2016solutions} for this simplified
situation as follows:

\begin{theorem}[Bachiller-Cedó-Jespers]
Let $B$ be a brace, $x\in B$ such that the set $\{\lambda_b(x):x\in B\}$ generate
the additive group of brace, and $K<B$ a core-free subgroup of the multiplicative group
such that $\lambda_k(x)=x\,\forall k\in K$. Then there is a natural solution
structure on $X/K$. This solution is indecomposable and its structure group is
isomorphic to $B$ as a brace.

Moreover any indecomposable solution with $B$ as its structure brace has this form.
\end{theorem}

In order to classify all solutions with a given permutation group, we must then
understand all possible braces with that multiplicative structure. For this we
refer to recent work classifying braces, and more general skew-braces, of fixed
sizes \cite{acri2020skew,acri2022abelian,dietzel2021braces,alabdali2021skew}.
We are going to be refering mainly to the results of Acri and Bonato in \cite{acri2020skew}
and \cite{acri2022abelian}, as they give explicit formulas for the structure of
the all the braces we require.

The process of classification is then the following. Using the explicit description
of the braces with fixed multiplicative group we first characterize all generting
orbits of the $\lambda$ action. In the case of abelian groups this gives us all
solutions, for the dihedral groups we also need to find the core-free
subgroups of the stabilizers to find all solutions. Next we need to find how the
brace automorphisms act on the orbits. For this we again use the explicit description
to compute the group of automorphisms explicitly. We the explicit description of
the orbits and automorphisms we can see which orbits generate isomorphic solutions.

\section{Braces of size $pq$}\label{sec:size-pq}

We first focus on solutions with permutation group of size $pq$. For the descriptions
of the braces we refer to \cite{acri2020skew}, where all (skew) braces of size $pq$ are
classified. According to the classification there are at most two braces of these
sizes. There is always a trivial brace with cyclic multiplicative group. When
$p\equiv 1\pmod{q}$ there is an additional brace whose multiplicative group is
a semidirect product.

We first consider the trivial brace. In this case, we give all
indecomposable solutions for any trivial brace,
\begin{proposition}\label{prop:trivial-braces}
Let $B$ be a trivial brace. If the underlying group of $B$ is cyclic then there
is a single indecomposable solution with associated brace $B$. If the underlying
group is not cyclic then there is no indecomposable solution.
\end{proposition}
\begin{proof}
By \cite{BCJ2016solutions} we first need to find all orbits under the $\lambda$
action that generate the additive group. Since the brace is trivial all orbits
are singletons, and the problem reduces to finding generators. In particular only
cyclic groups will produce any indecomposable solutions. We can then safely
assume the group is cyclic to find all such solutions.

Since the group is cyclic then only possible core-free subgroup is the trivial
group, so every generator of the group produces a single solution. However, since
the group is cyclic there is a group automorphism mapping every generator to any
other one.
Since the brace is trivial any additive morphism is a brace morphism, and since the subgroup
being used to construct the solution is trivial this gives an isomorphism
between any pair of the constructed solutions.
\end{proof}

Next we consider the non-trivial brace of size $pq$. For this we must have
\(p\equiv 1 \pmod{q}\). We identify the brace $B$ with its additive group
$\Z_{p}\times\Z_{q}$. The multiplicative structure is then given by
\begin{align}\label{eqn:mult-pq}
\begin{pmatrix} a \\ b \end{pmatrix}\circ\begin{pmatrix} c \\ d \end{pmatrix}
= \begin{pmatrix} a + g^{b} c \\ b + d \end{pmatrix},
\end{align}
where $g$ is any element of $\Z_{p}$ of order $q$.

By Remark~\ref{rmk:orbit-order}, any generating orbit must contain a generator, i.e. an
element of the form $(a,b)$ with both $a$ and $b$ generators of the corresponding
group. The stabilizer of any such orbit is then $\Z_{p}\times\{0\}$, giving
$\frac{(p-1)(q-1)}{q}$ distinct orbits. We note that this subgroup has no
non-trivial core-free subgroups.

To understand when two of these orbits generate the same solution we
need to know the group of brace automorphisms, these are given by the following
proposition:

\begin{proposition}\label{prop:auts-pq}
Let $p$ and $q$ be primes with \( p \equiv 1 \pmod{q}\). Then the group of
automorphisms of the non-trivial brace of size $pq$ is isomorphic to $\Z_p^{\times}$.
\end{proposition}
\begin{proof}
The automorphisms of the additive group are given by $\Z_p^{\times}\times\Z_q^{\times}$,
acting by coordinatewise multiplication. Given $(\alpha,\beta)$ in this group
for it to be a brace automorphism it must satisfy
\[
\begin{pmatrix}\alpha (a + g^b c)\\ \beta (b + d)\end{pmatrix}
= \begin{pmatrix} \alpha (a + g^{\beta b} c) \\ \beta (b + d) \end{pmatrix},
\qquad \forall a,c\in\Z_p, b,d\in\Z_q.
\]
And for that to hold the only possibility is \(\beta = 1\).
\end{proof}
From this it follows that two orbits will generate the same solution precisely
when the second coordinate of its elements coincide. This means that there are
exactly $q-1$ indecomposable solutions with this brace.

We then have the following result classifying all indecomposable solutions whose
permutation group has size $pq$:
\begin{theorem}\label{thm:clasif-pq}
Let $G$ be a group of order $pq$. Then:
\begin{enumerate}
  \item If $G$ is cyclic there is a single indecomposable solution with
  permutation group $G$.
  \item If $G$ is a semidirect product $\Z_p \rtimes \Z_q$ then there are
  exactly $q-1$ indecomposable solutions with permutation group $G$.
\end{enumerate}
In any of these cases all solutions have size $pq$.
\end{theorem}

A particular instance of this result is that if $p$ is a prime number then there
is a single indecomposable solution with permutation group the dihedral group
$D_{2p}$.

\section{Cyclic braces of size $p^{2}q$}

Now we turn to study braces of size $p^{2}q$ and start by considering those
whose multplicative group is cyclic. We refer to \cite{acri2022abelian} for
the enumaration and explicit descriptions of the braces.

\subsection{Case $p=2$}\label{sec:4q-cyc}

In \cite{acri2022abelian} the analysis is split into the cases \(q\equiv 1\pmod{4}\)
and \(q\equiv 3\pmod{4}\), however the braces with abelian permutation group
can be described uniformly for both cases so we make no such distinction here.

There are two possible braces with cyclic permutation group. The first is the
trivial brace. By Theorem~\ref{prop:trivial-braces} there is a single indecomposable
solution for this brace. For the other solution we identify the brace with its
additive group \(\Z_2^2\times\Z_q\). The multiplicative structure is
given by
\begin{align}\label{eqn:mult-cyc-4q}
\begin{pmatrix} a \\ b \\ c \end{pmatrix}\circ\begin{pmatrix} d \\ e \\ f \end{pmatrix}
= \begin{pmatrix} a + d + be \\ b + e \\ c + f \end{pmatrix}.
\end{align}

By Remark~\ref{rmk:orbit-order} the elements of any generating orbit must have
order $2q$. So we must consider the orbits of elements of the form
$(\alpha,\beta,\gamma)$ with $\gamma$ a generator of $\Z_q$ and
$\alpha$ and $\beta$ not both 0. Since acting via $\lambda$ on one of this elemnts
can only change its first coordinate we must have $\beta=1$. We then have
exactly $q-1$ generating orbits of the form
\(
\mathcal{O}_\gamma = \{ (*, 1, \gamma) \}
\), with $\gamma$ a generator of $\Z_q$.

The following result characterizes the group of brace automorphisms of these
braces:

\begin{proposition}\label{prop:auts-cyc-4q}
Given $q$ a prime number, the group of brace automorphisms
of the only non-trivial brace with multiplicative group isomorphic to
$\Z_{4q}$ is isomorphic to $\Z_2\times\Z_q^\times$, with
\end{proposition}
\begin{proof}
The group automorphisms of the additive group are isomorphic to
$\GL_2(2)\times\Z_q^\times$, with the first component acting by multiplication
on the first two coordinates, and the second component on the third one. Let
\((A,u)\) be an element of this group, with
\( A = {\tiny \begin{pmatrix} x & y \\ w & z \end{pmatrix} }\). We replace
each term by $(A,u)$ times the element in \eqref{eqn:mult-noCyc-p2q} to check
which ones are also multiplicative morphisms. Focusing on the second coordinate
we get
\[
w(a+d)+z(b+e) = w(a+d+be)+z(b+e),
\]
if this is to hold for any $a,b,d,e\in\Z_2$ then $w$ must be zero. This immediately
implies we must have $x=1=z$ for the matrix to be invertible. The condition on
the first coordinate is now automatically satisfied and from the third
coordinate we do not get any restriction, so for $(A,u)$ to be a brace automorphism
$A$ must be unitriangular. Since the group of unitriangular matrices is isomorphic
to $\Z_2$ we get the desired result.
\end{proof}

Given two generating orbits $\mathcal{O}_\gamma$ and $\mathcal{O}_{\gamma'}$ we
can take $A$ the identity matrix and $u=\gamma^{-1}\gamma'$ to get by the previous
result a brace automorphism mapping the first orbit to the second one.
In particular all the orbits generate the same solution, i.e. there is a single
indecomposable solution with non-trivial brace in this case.

\subsection{Case $p$ odd}

In this case we again have only two possible braces, the trivial one and a single
non-trivial one. However, in this case the non-trivial brace also has cyclic additive
group. Identifying the brace with its additive group \(\Z_{p^2}\times\Z_q\), the
multiplicative structure is given by
\begin{align}\label{eqn:mult-cyc-p2q}
\begin{pmatrix} a \\ b \end{pmatrix} \circ \begin{pmatrix} c \\ d \end{pmatrix}
= \begin{pmatrix} a + c + p a c \\ b + d \end{pmatrix}.
\end{align}
By Remark~\ref{rmk:orbit-order} any generating orbit is the orbit of a generator
of the additive group. Given $(\alpha,\beta)$ a generator of $\Z_{p^2}\times\Z_q$
its orbit under the $\lambda$ action consists of all elements of the form
$(\hat{\alpha},\beta)$ with $\hat{\alpha}\equiv\alpha\pmod{p}$.

To see which of this orbits generate the same solution we characterize the group
of brace automorphisms of this brace.

\begin{proposition}\label{prop:auts-cyc-p2q}
The group of brace automorphisms of the only non-trivial brace of order $p^2q$
with cyclic multiplicative group is isomorphic to $\Z_{q}^\times$.
\end{proposition}
\begin{proof}
The group of automorphisms of the additive group is $\Z_{p^2}^\times\times\Z_q^\times$,
acting by coordinatewise multiplication. From \eqref{eqn:mult-cyc-p2q} we can
see that for an element $(\alpha,\beta)$ to also be a morphism of the
multiplicative group it must satisfy \(\alpha^2\equiv\alpha\pmod{p^2}\).
\end{proof}

From this result we then conclude that two generators of the additive group
$(\alpha,\beta)$ and $(\alpha',\beta')$ give the same solution if
$\alpha\equiv\alpha'\pmod{p}$. In particular, there are $p-1$ indecomposable
solutions for this brace.

Together with the previous case we have proved the following theorem:
\begin{theorem}\label{thm:calsif-p2q-cyc}
Let $p$ and $q$ be distinct prime numbers. There are $p$ distinct indecomposable
solutions with permutation group isomorphic to $\Z_{p^2q}$.
\end{theorem}

\section{Non-Cyclic abelian braces of size $p^2q$}

We now focus on the only non-cyclic abelian group of order $p^2q$. We note that
by Proposition~\ref{prop:trivial-braces} the trivial brace will not give us a
solution in this case.

\subsection{Case $p = 2$}

In this case there is a single non-trivial brace and its additive group is cyclic.
Identifying the brace with its additive group $\Z_q\times\Z_4$ it multiplicative
structure is given by
\begin{align}\label{eqn:mult-noCyc-4q}
\begin{pmatrix} a \\ b \end{pmatrix} \circ \begin{pmatrix} c \\ d \end{pmatrix}
=\begin{pmatrix} a + c \\ b + (-1)^{b} d \end{pmatrix}.
\end{align}
The generating orbits are then of the form
\(\mathcal{O}_\alpha = \{(\alpha,1),(\alpha,3)\} \), with $\alpha$ a generator of
$\Z_q$.

We can characterize the group of brace automorphisms with the following result:
\begin{proposition}
Given $q$ an odd prime, the only non-trivial brace with mutliplicative group
$\Z_q\times\Z_2^2$ has group of automorphisms isomorphic to
$\Z_{4q}^\times$.
\end{proposition}
\begin{proof}
The gorup of automorphisms of the additive group is
\(\Z_{4q}^\times\cong\Z_q^\times\times\Z_4^\times\), which acts by coordinatewise
multiplication. By substituting on \eqref{eqn:mult-noCyc-4q} we see that all of
these are also multiplicative automorphisms.
\end{proof}

It follows from this result that all orbits generate the same solution. There is
then a single indecomposable solution with permutation group $\Z_q\times\Z_2^2$
for any odd prime $q$.

\subsection{Case $p$ odd}

Like in the previous case there is a single non-trivial brace, however in this
case the multiplicative group is isomorphic to $\Z_p^2\times\Z_q$. Identifying
the brace with the additive group, the multiplicative structure is given by
\begin{align}\label{eqn:mult-noCyc-p2q}
\begin{pmatrix} a \\ b \\ c \end{pmatrix} \circ \begin{pmatrix} d \\ e \\ f \end{pmatrix}
=\begin{pmatrix} a + d + be \\ b + e \\ c + f \end{pmatrix}.
\end{align}

Notice that this is the same formula that defines the non-trivial brace of
subsection~\ref{sec:4q-cyc}. However when $p$ is an odd prime the resulting
multiplicative group is not cyclic. The orbit of an element $(\alpha,0,\beta)$
only contains elements of the same form, and in particular cannot generate the
additive group. The rest of the orbits are of the form
\(
\mathcal{O}_{\alpha,\beta} = \{ (*,\alpha,\beta) \}
\), with \(\alpha \neq 0\).
For such an orbit to generate both $\beta$ must also be non-zero. The
following proposition characterizes the brace automorphisms:
\begin{proposition}
Given $p$ an odd prime and $q$ a prime number, the group of brace automorphisms
of the only non-trivial brace with multiplicative group isomorphic to
$\Z_p^2\times\Z_q$ is isomorphic to $G\times\Z_q^\times$, with
\[
G = \left\{ \begin{pmatrix} x^2 & y \\ 0 & x \end{pmatrix}
            : x \in \Z_p^\times, y \in \Z_p \right\} \subset \GL_2(p).
\]
\end{proposition}
\begin{proof}
We procede like in Proposition~\ref{prop:auts-cyc-4q}.
The group automorphisms of the additive group are isomorphic to
$\GL_2(p)\times\Z_q^\times$, with the first component acting by multiplication
on the first two coordinates, and the second component on the third one. Let
\((A,u)\) be an element of this group, with
\( A = {\tiny \begin{pmatrix} x & y \\ w & z \end{pmatrix} }\). We act on
each term of \eqref{eqn:mult-noCyc-p2q} by $(A,u)$ to check
which ones are also multiplicative morphisms. Focusing on the second coordinate
we get
\[
w(a+d)+z(b+e) = w(a+d+be)+z(b+e),
\]
and like before for this to hold for any $a,b,d,e\in\Z_p$ $w$ must be zero. In
this case the condition on the first coordinate is not automatically satisfied.
Using $w=0$ we get
\[
x(a+d)+y(b+e)+z^2be = x(a+d)+y(b+e)+xbe,
\]
that will hold for any $a,b,d,e\in\Z_p$ if and only if $x=z^2$. From the third
coordinate we do not get any restriction, so for $(A,u)$ to be a brace automorphism
it must be of the desired form.
\end{proof}

Given two generating orbits $\mathcal{O}_{\alpha,\beta}$ and
$\mathcal{O}_{\alpha',\beta'}$ we can take \(u = \beta^{-1}\beta' \) and
\(
A={\tiny \begin{pmatrix} (\alpha^{-1}\alpha')^2 & 0 \\ 0 &\alpha^{-1}\alpha' \end{pmatrix}}
\) to get an
automorphism that maps the first orbit to second one. This means that all the
orbits generate the same solution.

With this and the previous case we have proved the following theorem:
\begin{theorem}
Given $p$ and $q$ prime numbers, there is a single indecomposable solution with
permutation group $\Z_p^2\times\Z_q$.
\end{theorem}

\section{Dihedral braces of size $p^2q$}

In this last section we turn our attention to solutions with dihedral permutation
group. With the results of \cite{acri2022abelian} we can get all solutions with
permutation group $D_{2p^2}$ or $D_{4p}$ for $p$ an odd prime. We first note that
in any dihedral group every rotation generates a normal subgroup, and any two
distinct reflections generate some rotation. In particular, in any dihedral group
the only core-free subgroups are those generated by a single reflection, and so
we have the following result:

\begin{theorem}\label{prop:size-dih-sols}
If $X$ is an indecomposable solution with permutation group $D_{2n}$, then $|X|=2n$
or $|X|=n$.
\end{theorem}

\subsection{Case $2p^2$}
For this case we will analyze a slightly more general one, that of a semidirect
product $\Z_{p^2}\rtimes\Z_q$ when \(p\equiv 1\pmod{q}\). Taking $q=2$ we get
the desired dihedral group. These semidirect products are given by choosing
$g$ an element of order $q$ in $\Z_{p^2}^\times$, and letting a generator of
$\Z_q$ act by mutliplication by $g$.

There is a single brace
with this group as its multpiplicative group, and it has additive group isomorphic
to the cyclic group. Its multiplicative structure is given by
\[
\begin{pmatrix} a \\ b \end{pmatrix} \circ \begin{pmatrix} c \\ d \end{pmatrix}
= \begin{pmatrix} a + g^b c \\ b + d \end{pmatrix}.
\]

We note that this formula is precisely the same as \eqref{eqn:mult-pq}, and with
the same arguments as in Section~\ref{sec:size-pq} we get that the automorphism
group of this braces is given by the following result:
\begin{proposition}\label{prop:auts-pq}
Let $p$ and $q$ be primes with \( p \equiv 1 \pmod{q}\). Then the group of
automorphisms of the non-trivial brace of size $p^2q$ is isomorphic to
$\Z_{p^2}^{\times}$.
\end{proposition}

Like before the generating orbits are the ones generated by elements $(\alpha,\beta)$
with $\alpha$ and $\beta$ generators, and two elements give the same solution
precisely when the have the same second coordinate.
\begin{theorem}
Given $p$ and $q$ primes such that $p\equiv 1 \pmod{q}$, there are exactly
$q-1$ indecomposable solutions with permutation group $\Z_{p^2}\times\Z_{q}$,
and all such solutions have size $p^2q$.
\end{theorem}
In particular we get the following consequence:
\begin{corollary}
Given $p$ an odd prime, there is a single indecomposable solution with permutation
group the dihedral group $D_{2p^2}$. This solution has size $2p^2$.
\end{corollary}

\subsection{Case $4p$}

In this case we have several brace structures to study. There are two braces
with cyclic additive group and another one with non-cyclic additive group. The
first cyclic one has multiplicative structure given by
\begin{align}\label{eqn:mult-dih-1}
\begin{pmatrix} a \\ b \end{pmatrix} \circ \begin{pmatrix} c \\ d \end{pmatrix}
= \begin{pmatrix} a + (-1)^b c \\ b + (-1)^b d \end{pmatrix}.
\end{align}
The orbit of a generator $(\alpha,\beta)$ consists of itself and the element
$(-\alpha,-\beta)$, and its stabilizer is $\{(*,0),(*,2)\}$, which is precisely
the group of all rotations, and so has no non-trivial core-free subgroups. The
group of brace automorphisms is given by the following result:
\begin{proposition}
Given $p$ an odd prime, the group of brace automorphisms of the brace with cyclic
additive group and multiplicative group given by \eqref{eqn:mult-dih-1} is
isomorphic to $\Z_{4p}^\times$.
\end{proposition}
\begin{proof}
An automorphism of the additive group is given by multiplication by a pair of
elements $(x,y)$, and one can verify that all of them are brace automorphisms.
\end{proof}
In particular this means that all the orbits give rise to the same solution.

The second brace with cyclic additive group is given by the formula
\begin{align}\label{eqn:mult-dih-2}
\begin{pmatrix} a \\ b \end{pmatrix} \circ \begin{pmatrix} c \\ d \end{pmatrix}
=\begin{pmatrix} a + (-1)^{b(b-1)/2} c \\ b + (-1)^b d \end{pmatrix}.
\end{align}
In this case the orbit of a generator has four elements, and its stabilizer is
given by the subgroup $\{(*,0)\}$. This group consists entirely of rotations and
so has no non-trivial core-free subgroups. We characterize the group of brace
automorphsims with the following resul:
\begin{proposition}
Given $p$ an odd prime, the group of brace automorphisms of the brace with cyclic
additive group and multiplicative group given by \eqref{eqn:mult-dih-2} is
isomorphic to $\Z_{p}^\times$.
\end{proposition}
\begin{proof}
An automorphism of the additive group is given by multiplication by a pair of
elements $(x,y)$. Acting with this in both sides of the formula \eqref{eqn:mult-dih-2},
we get that to be a brace automorphism the pair must satisfy $y=1$, which gives us the desired
result.
\end{proof}
With this we can see that all the orbits generate the same solution.

Finally we look at the non-cyclic brace. Identifying it with its additive group
$\Z_2^2\times\Z_p$ its multiplicative structure is given by
\begin{align}\label{eqn:mult-dih-3}
\begin{pmatrix} a \\ b \\ c \end{pmatrix}\circ\begin{pmatrix} d \\ e \\ f \end{pmatrix}
=\begin{pmatrix} a + d  \\ b + e \\ c + (-1)^a f \end{pmatrix}.
\end{align}

The lambda action on an element does not modify the first two coordinates. Concretely
this means the orbit of an element $(\alpha,\beta,\gamma)$ is
$\{(\alpha,\beta,\pm\gamma)\}$, and in particualar it cannot
generate the additive group. So this brace has no associated solutions.

With this result we can see that all orbits are conjugate under isomorphism and
so generate the same solution. In conclusion we get the following result:

\begin{theorem}
Given $p$ an odd prime, there are two indecomposable solutions with permutation
group isomorphic to $D_{4p}$. Moreover all of these solutions have size $4p$.
\end{theorem}

As we noted in Proposition~\ref{prop:size-dih-sols} an indecomposable solution
with permutation group $D_{2n}$ can only have size $n$ or $2n$. From the results
obtained here we conjecture the first case never happens:
\begin{conjecture}
If $X$ is an indecomposable solution with permutation group $D_{2n}$ then
$|X|=2n$.
\end{conjecture}

\bibliographystyle{abbrv}
\bibliography{refs}

\end{document}